\documentclass[twoside,12pt]{article}
\usepackage[all, 2cell, dvips]{xy}
\usepackage{latexsym,amsmath, amsfonts, graphics, epsf, epic}
\usepackage{amsthm}
\usepackage{amssymb}
\usepackage{amsopn}
\usepackage{amscd}
\usepackage{color}

\theoremstyle{plain}
\newtheorem{thm}{Theorem}[section]

\newtheorem{prop}[thm]{Proposition}

\theoremstyle{definition}

\newtheorem{remark}[thm]{Remark}
\newtheorem{example}[thm]{Example}
\newtheorem{defn}[thm]{Definition}

\usepackage{amssymb}

\def\lm{\lambda}

\def\al{\alpha}

\def\l.l.o.{\it l.l.o}

\def\chiup{\raise 2pt\hbox{$\chi$}}

\title{Asymptotics of Young tableaux in the strip, the $d$-sums}

\author{ A. Regev}

\begin{document}

%{\bf Abstract}.

\maketitle

{\bf Abstract}. The asymptotics of the "strip" sums
$S_\ell^{(\al)}(n)$ and of their $d$-sums generalizations
$T_{d,ds}^{(\al)}(dm)$ (see Definition~\ref{definition1}) were
calculated in~\cite{regev}. It was recently noticed that when
$d>1$ there is a certain confusion about the relevant notations
in~\cite{regev}, and the constant in the asymptotics of these
$d$-sums $T_{d,ds}^{(\al)}(dm)$ seems to be off by a certain
factor. Based on the techniques of~\cite{regev} we again calculate
the asymptotics of the $d$-sums $T_{d,ds}^{(\al)}(dm)$. We do it
here carefully and with complete details. This leads to
Theorem~\ref{d.sum222} below, which replaces Corollary 4.4
of~\cite{regev} in the cases $d>1$.


\section{Introduction}
Let $\lm$ be a partition and $\ell(\lm)$  the number of non-zero
parts of $\lm$. Let $f^\lm$ denote the number of standard tableaux
of shape $\lm$. For the Young-Frobenius formula for $f^\lm$ see
for example~\cite[2.3.22]{jameskerber}, and for the {\it "hook"}
formula see for example~\cite[corollary 7.21.5]{stanley}.

\medskip
The asymptotics of the sums $S_\ell^{(\al)}(n)$ and of the
$d$-sums $T_{d,ds}^{(\al)}(dm)$ (see Definition~\ref{definition1})
were studied in~\cite{regev}, see~\cite[Corollary 4.4]{regev}
 (there we used the
notation $d_\lm$ instead of $f^\lm$). We recently noticed that
when $d>1$ there is a certain confusion about the notations
in~\cite{regev}, and the constant in the asymptotics of the
$d$-sums $T_{d,ds}^{(\al)}(dm)$ seems to be off by a certain
factor.

\medskip
Based on the techniques of~\cite{regev} we calculate,
with complete details, the
asymptotics of the $d$-sums $T_{d,ds}^{(\al)}(dm)$. While ths asymptotic
formula  for the
sums $S_\ell^{(\al)}(n)$ remain unchanged as in~\cite{regev}, this
leads to a new asymptotic formula for the $d$-sums
$T_{d,ds}^{(\al)}(dm)$, given in Theorem~\ref{d.sum222} below.

\medskip
The validity of Theorem~\ref{d.sum222} can be tested as follows.
In few cases the $d$-sums $T_{d,ds}^{(\al)}(dm)$ are 
given by a closed formula, which yield the corresponding asymptotics directly--
independent of Theorem~\ref{d.sum222}. In all these cases, the
direct asymptotics and the asymptotics deduced from
Theorem~\ref{d.sum222} -- agree, see Section~\ref{3.1}. Also,
for small values of $d$ and $s$ it is possible to write an explicit formula for, say, $T_{d,ds}^{(1)}(dm)$. By Theorem~\ref{d.sum222}
$T_{d,ds}^{(1)}(dm)\simeq A(d,s,dm)$. Now form the ratio
$T_{d,ds}^{(1)}(dm)/ A(d,s,dm)$. Using, say, "Mathematica", calculate that ratio
for increasing values of $m$, verifying that these values become closer and closer to 1 as $m$ increases. This again tests and indicates the validity of Theorem~\ref{d.sum222}.

\subsection{The main theorem}

The following definition recalls the $d$-sums from~\cite{regev}.

\begin{defn}\label{definition1}
Let $m,s,d\ge 1$, then define
\begin{enumerate}
\item
\[
B_d(dm)=\{\lm\vdash dm\mid d\quad\mbox{divides all}\quad \lm_j'\}.
\] Note that $\lm\in B_d(dm)$ if and only if $\lm$ can be written as
$\lm=(\mu_1^d,\mu_2^d,\ldots)$ with $(\mu_1,\mu_2,\ldots)\vdash
m$, and then $d$ divides $\ell(\lm)$.
\item
\[
B_{d,ds}(dm)=\{\lm\in B_d(dm)\mid \ell(\lm)\le
ds\}\qquad\mbox{and}
\]
\item
\[
T_{d,ds}^{(\al)}(dm)=\sum_{\lm\in B_{d,ds}(dm)}(f^\lm)^\al.
\]
\item
When $d=1$ we denote $T_{1,s}^{(\al)}(m)=S_s^{(\al)}(m)$. Thus
\[
S_s^{(\al)}(m)=\sum_{\lm\vdash m,~\ell(\lm)\le s} (f^\lm)^\al.
\]

\end{enumerate}
\end{defn}

We correct~\cite[Corollary 4.4]{regev} in the case $d>1$ by
proving the following theorem (see Theorem~\ref{d.sum2} below).
Here the variable $N$ is replaced by $s$.

\begin{thm}\label{d.sum222}
Let $1\le d,s\in\mathbb{Z}$ and let $0<\al\in\mathbb{R}$. As
$m\to\infty$,
\[
T_{d,ds}^{(\al)}(dm)\simeq~~~~~~~~~~~~~~~~~~~~~~~~~~~~~~~~~~~~~~~~~~~~~~~~~~~~~~~~~~~~~~~~~~~~~~~~~~~~~~~~~~~~~~~~~~
\]
\[
\simeq\left[ \left( \frac{1}{\sqrt{2\pi}} \right)^{ds-1}\cdot
\sqrt d\cdot s^{d^2s^2/2}\cdot (2!\cdots (d-1)!)^s\cdot
\left(\frac{1}{\sqrt{m}} \right)^{(d^2s^2+d^2s-2)/2} \cdot
(ds)^{dm}\right]^\al\cdot
\]
\[~~~~~~~~~~~~~~~~~~~\cdot(\sqrt{m})^{s-1} \cdot\left(\frac{d}{s}\right)^{(s-1)(\al
s+2)/4}\cdot\frac{d}{\sqrt
s}\cdot\sqrt{\frac{\al}{2\pi}}\cdot\frac{1}{s!}\cdot~~~~~~~~~~~~
\]
\[
~~~~~~~~~~~~~~~~~~~~~~~~~~~~~~~\cdot(2\pi)^{s/2}\cdot(d^2\al)^{-s/2-d^2\al
s(s-1)/4}\cdot (\Gamma(1+d^2\al/2))^{-s}\cdot\prod_{j=1}^s
\Gamma(1+d^2 \al j/2).
\]
\end{thm}

\section{Asymptotics for a single $f^\lm$}

The following proposition corrects (and
replaces)~\cite[(F.1.3)]{regev}, and is the key for proving
Theorem~\ref{d.sum222}. Recall the notation \[
D_s(x_1,\ldots,x_s)=\prod_{1\le i<j\le s}(x_i-x_j).
\]
\begin{prop}\label{sch9}
Let $\lm=(\lm_1^d,\ldots,\lm_s^d)\vdash dm=n$. For $1\le i\le s$
write $\lm_i=m/s+b_i\sqrt m$ and assume the $b_i$ are bounded, so
$\lm_i\simeq m/s$. Then, as $m$ goes to infinity,

\[
f^\lm\simeq\left( \frac{1}{\sqrt{2\pi}} \right)^{ds-1}\cdot \sqrt
d\cdot s^{d^2s^2/2}\cdot (2!\cdots (d-1)!)^s\cdot
\left(\frac{1}{\sqrt{m}} \right)^{(d^2s^2+d^2s-2)/2} \cdot
(ds)^{dm}\cdot~~~~~
\]
\[
~~~~~~~~~~~~~~~~~~~~~~~~~~~~~~~~~~~~~~~~~~~~~~~~~~~~~~~~~~ \cdot
D_s(b_1,\ldots,b_s)^{d^2}\cdot e^{-(ds/2)(b_1^2+\cdots +b_s^2)}=
\]
\[
=\left( \frac{1}{\sqrt{2\pi}} \right)^{ds-1}\cdot
d^{(d^2s^2+d^2s)/4}\cdot s^{d^2s^2/2}\cdot (2!\cdots
(d-1)!)^s\cdot \left(\frac{1}{\sqrt{dm}}
\right)^{(d^2s^2+d^2s-2)/2} \cdot (ds)^{dm}\cdot~~~~~
\]
\[
~~~~~~~~~~~~~~~~~~~~~~~~~~~~~~~~~~~~~~~~~~~~~~~~~~~~~~~~~~ \cdot
D_s(b_1,\ldots,b_s)^{d^2}\cdot e^{-(ds/2)(b_1^2+\cdots +b_s^2)}.
\]
\end{prop}

\begin{proof}
 Apply, for example, the
Young-Frobenius formula for $f^\lm$: First, all $\lm_i\simeq m/s$,
hence we can write
\begin{eqnarray}\label{sch044}
f^\lm\simeq\left(\frac{s}{m}\right)^{ds(ds-1)/2}
\cdot\frac{(dm)!}{(\lm_1!)^d\cdots (\lm_s!)^d} \cdot
H(\lm_1,\ldots,\lm_s)
\end{eqnarray}
where $H(\lm_1,\ldots,\lm_s)$ is the product of factors of the
form $\lm_i-\lm_j+k$, with various $0\le k\le ds$, and which we
now analyze.

\medskip
 For $1\le i<j\le s$ there are
$d^2$ factors of $f^\lm$ of the form $\lm_i-\lm_j+k$, with various
$k$'s, all of them  satisfying $\lm_i-\lm_j+k\simeq (b_i-b_j)\sqrt
m$. The number of pairs $(i,j)$ where $1\le i<j\le s$ is
$s(s-1)/2$, and each such pair contributes $d^2$ times the factor
$(b_i-b_j)\sqrt m$ , hence the factor $D_s(b_1,\ldots,
b_s)^{d^2}\cdot(\sqrt m)^{d^2s(s-1)/2}$ in~\eqref{sch4} below.

\medskip

In the cases $i=j$ each of the $s$ blocks $(\lm_i^d)$ contributes
$D_d(d,d-1,\ldots,1)=1!\cdot 2!\cdots (d-1)!$, hence the factor
$(1!\cdot 2!\cdots (d-1)!)^s$ in~\eqref{sch4} below. It follows
that
\begin{eqnarray}\label{sch4}
f^\lm\simeq\left(\frac{s}{m}\right)^{ds(ds-1)/2}\cdot (2!\cdots
(d-1)!)^s\cdot D_s(b_1,\ldots, b_s)^{d^2} \cdot(\sqrt
m)^{d^2s(s-1)/2}\cdot\frac{(dm)!}{(\lm_1!)^d\cdots (\lm_s!)^d} .
\end{eqnarray}
Again since $\lm_i\simeq m/s$,
\begin{eqnarray}\label{sch44}
\frac{m!}{(\lm_1+s-1)!\cdots (\lm_s)!}\simeq
 \left(\frac{s}{m} \right)^{s(s-1)/2}\cdot  \frac{m!}{(\lm_1!)\cdots
 (\lm_s!)}.
\end{eqnarray}
By~\cite[Step 3, page 118, with $\sqrt{2\pi}$ replacing and
correcting $2\pi$]{regev}
\begin{eqnarray}\label{sch45}
\frac{m!}{(\lm_1+s-1)!\cdots (\lm_s)!}\simeq \left(\frac{1}{\sqrt
{2\pi}} \right)^{s-1}\cdot s^{s^2/2}\cdot \left(\frac{1}{ {m}}
\right)^{(s^2-1)/2}\cdot s^m\cdot e^{-(s/2)(b_1^2+\cdots +b_s^2)},
\end{eqnarray}
hence by~\eqref{sch44} and~\eqref{sch45}
\[
\frac{m!}{(\lm_1!)\cdots
 (\lm_s!)}\simeq\left(\frac{m}{s} \right)^{s(s-1)/2}\cdot \left(\frac{1}{\sqrt
{2\pi}} \right)^{s-1}\cdot s^{s^2/2}\cdot \left(\frac{1}{ {m}}
\right)^{(s^2-1)/2}\cdot s^m\cdot e^{-(s/2)(b_1^2+\cdots +b_s^2)}=
\]
\begin{eqnarray}\label{sch5}
=\left(\frac{1}{\sqrt {2\pi}} \right)^{s-1}\cdot
s^{s/2}\cdot\left(\frac{1}{m} \right)^{(s-1)/2}\cdot s^m\cdot
e^{-(s/2)(b_1^2+\cdots +b_s^2)}.
\end{eqnarray}
Now
\begin{eqnarray}\label{sch6}
\frac{(dm)!}{(\lm_1!)^d\cdots (\lm_s!)^d}\simeq
\frac{(dm)!}{(m!)^d}\cdot\left(\frac {m!}{\lm_1!\cdots\lm_s!}
\right)^d
\end{eqnarray}
and by Stirling's formula
\begin{eqnarray}\label{sch7}
\frac{(dm)!}{(m!)^d}\simeq \left( \frac{1}{\sqrt{2\pi}}
\right)^{d-1}\cdot \sqrt d\cdot\left( \frac{1}{\sqrt{m}}
\right)^{d-1}\cdot d^{dm}.
\end{eqnarray}
It follows from~\eqref{sch5},~\eqref{sch6} and~\eqref{sch7} that
\[
 \frac{(dm)!}{(\lm_1!)^d\cdots
(\lm_s!)^d}\simeq~~~~~~~~~~~~~~~~~~~~~~~~~~~~~~~~~~~~~~~~~~~~~~~~~~~~~~~~~~~~~~~~~~~~~~~~~~~~~~
\]
\[
\simeq\left[\left( \frac{1}{\sqrt{2\pi}} \right)^{d-1}\cdot \sqrt
d\cdot\left( \frac{1}{\sqrt{m}} \right)^{d-1}\cdot
d^{dm}\right]\cdot \left[ \left(\frac{1}{\sqrt {2\pi}}
\right)^{s-1}\cdot s^{s/2}\cdot\left(\frac{1}{m}
\right)^{(s-1)/2}\cdot s^m\cdot e^{-(s/2)(b_1^2+\cdots +b_s^2)}
\right]^d=
\]
\begin{eqnarray}\label{sch8}
~~~~~~~~~~~~~~~~~~~~~~~=\left( \frac{1}{\sqrt{2\pi}}
\right)^{ds-1}\cdot \sqrt d\cdot s^{ds/2}\cdot\left(
\frac{1}{\sqrt{m}} \right)^{ds-1}\cdot (ds)^{dm}\cdot
e^{-(ds/2)(b_1^2+\cdots +b_s^2)}.
\end{eqnarray}
Together with~\eqref{sch4}
 this yields
\[
f^\lm\simeq\left[\left(\frac{s}{m}\right)^{ds(ds-1)/2}\cdot
D_s(b_1,\ldots, b_s)^{d^2}\cdot (2!\cdots (d-1)!)^s\cdot(\sqrt
m)^{d^2s(s-1)/2}\right]\cdot~~~~~~~~~~~~~~~~~~~~~~~~
\]
\[
~~~~~~~~~~~~~~~\cdot\left[\left( \frac{1}{\sqrt{2\pi}}
\right)^{ds-1}\cdot \sqrt d\cdot s^{ds/2}\cdot\left(
\frac{1}{\sqrt{m}} \right)^{ds-1}\cdot (ds)^{dm}\cdot
e^{-(ds/2)(b_1^2+\cdots +b_s^2)}\right]=
\]
\[
=\left( \frac{1}{\sqrt{2\pi}} \right)^{ds-1}\cdot \sqrt d\cdot
s^{d^2s^2/2}\cdot (2!\cdots (d-1)!)^s\cdot
\left(\frac{1}{\sqrt{m}} \right)^{(d^2s^2+d^2s-2)/2} \cdot
(ds)^{dm}\cdot~~~~~~~~
\]
\[
~~~~~~~~~~~~~~~~~~~~~~~~~~~~~~~~~~~~~~~~~~~~~~~~~~ \cdot
D_s(b_1,\ldots,b_s)^{d^2}\cdot e^{-(ds/2)(b_1^2+\cdots +b_s^2)}.
\]

This completes the proof of Proposition~\ref{sch9}.
\end{proof}

\subsection{Some examples}

\begin{example}
Using "Mathematica",  Proposition~\ref{sch9} was tested and
confirmed in the case $d=3,~s=2,~b_1=1$ and $b_2=-1$, and with
$n=3m$ getting larger and larger.
\end{example}

\begin{example}
The case $s=1$, any $d$, so $\lm =(m,\ldots,m)=(m^d)$. In this
case
\[
f^\lm=\frac{(dm)!\cdot 2!\cdots(d-1)!}{m!\cdot (m+1)!\cdots
(m+d-1)!}.
\]
By applying Stirling's formula directly we get that as $m\to
\infty$,
\[
f^\lm\simeq \left(\frac{1}{\sqrt{2\pi}} \right)^{d-1}\cdot
2!\cdots (d-1)!\cdot \sqrt d\cdot \left(\frac{1}{\sqrt{m}}
\right)^{d^2-1}\cdot d^{dm}.
\]
\end{example}

This agrees with Proposition~\ref{sch9} since the factor
$D_s(b_1,\ldots,b_s)^{d^2}\cdot e^{-(ds/2)(b_1^2+\cdots +b_s^2)}$
in that proposition equals 1 in this case.

\begin{example}
Here we repeat the proof of Proposition~\ref{sch9}
 - in the case $d=s=2$, showing more explicitly the various steps
 of the calculations. Let $\lm=(\lm_1,\lm_1,\lm_2,\lm_2)\vdash
2m$, so $(\lm_1,\lm_2)\vdash m$. Let $\lm_j=\frac{m}{2}+b_j\sqrt
m\simeq \frac{m}{2}$. In that case we verify directly that
\begin{eqnarray}\label{sch3}
f^\lm\simeq \left(\frac{1}{\sqrt{2\pi}}\right)^{3}\cdot
2^{14}\cdot\left(\frac{1}{\sqrt{2m}} \right)^{11} 4^{2m}\cdot
(b_1-b_2)^4\cdot e^{-2(b_1^2+b_2^2)}.
\end{eqnarray}

\begin{proof}
By either the hook formula or by the Young-Frobenius formula
\[
f^\lm=\frac{(2m)! \cdot  (\lm_1-\lm_2+1)\cdot
(\lm_1-\lm_2+2)^2\cdot (\lm_1-\lm_2+3)}{(\lm_1+3)!\cdot (\lm_1+2)!
\cdot  (\lm_2+1)! \cdot  \lm_2!}.
\]
Also $\lm_i+j\simeq m/2$ while $\lm_1-\lm_2+j\simeq (b_1-b_2)\sqrt
m$, hence
\begin{eqnarray}\label{sch1}
f^\lm\simeq \left(\frac{2}{m} \right)^6 \cdot  (b_1-b_2)^4 \cdot
m^2\cdot  \frac{(2m)!}{(\lm_1!)^2 \cdot  (\lm_2!)^2}.
\end{eqnarray}

By~\cite["Step 3" with $\sqrt{2\pi}$ replacing ${2\pi}$ (page
118)]{regev}
\[
\frac{m!}{(\lm_1+1)!\cdot\lm_2}\simeq\frac{1}{\sqrt{2\pi}}\cdot
4\cdot 2^m\cdot\left(\frac{1}{m}  \right)^{3/2}\cdot
e^{-(b_1^2+b_2^2)}.
\]
Since $\lm_1+1\simeq m/2$,
\[
\frac{m!}{\lm_1!\cdot\lm_2!}\simeq\frac{m}{2}\cdot
\frac{1}{\sqrt{2\pi}}\cdot 4\cdot 2^m\cdot\left(\frac{1}{m}
\right)^{3/2}\cdot e^{-(b_1^2+b_2^2)}= \frac{1}{\sqrt{2\pi}}\cdot
2\cdot 2^m\cdot\left(\frac{1}{m} \right)^{1/2}\cdot
e^{-(b_1^2+b_2^2)}.
\]
Also
\[
\frac{(2m)!}{(m!)^2}\simeq\frac{\sqrt 2}{\sqrt{2\pi}}\cdot
\frac{1}{\sqrt m}\cdot 2^{2m}.
\]
Thus
\[
\frac{(2m)!}{(\lm_1!)^2\cdot
(\lm_2!)^2}=\left(\frac{m!}{\lm_1!\cdot \lm_2!} \right)^2\cdot
\frac{(2m)!}{(m!)^2}\simeq
\]
\[
\left[  \frac{1}{\sqrt{2\pi}}\cdot 2\cdot
2^m\cdot\left(\frac{1}{m} \right)^{1/2}\cdot
e^{-(b_1^2+b_2^2)}\right]^2\cdot\left[\frac{\sqrt
2}{\sqrt{2\pi}}\cdot \frac{1}{\sqrt m}\cdot 2^{2m}\right]
\]
namely
\begin{eqnarray}\label{sch2}
\frac{(2m)!}{(\lm_1!)^2\cdot (\lm_2!)^2}\simeq
\left(\frac{1}{\sqrt{2\pi}}\right)^{3}\cdot 4\cdot \sqrt 2\cdot
4^{2m}\cdot\left(\frac{1}{m} \right)^{3/2}\cdot
e^{-2(b_1^2+b_2^2)}.
\end{eqnarray}
Finally
\[
f^\lm\simeq \left(\frac{2}{m} \right)^6\cdot m^2\cdot
(b_1-b_2)^4\cdot \frac{(2m)!}{(\lm_1!)^2\cdot
(\lm_2!)^2}\simeq~~~~~~~~~~~~~~~~~~~~~~~~~~~~~~~~~~~~~
\]
\[
~~~\simeq\left(\frac{2}{m} \right)^6\cdot m^2\cdot
(b_1-b_2)^4\cdot\left(\frac{1}{\sqrt{2\pi}}\right)^{3}\cdot 4\cdot
\sqrt 2\cdot 4^{2m}\cdot\left(\frac{1}{m} \right)^{3/2}\cdot
e^{-2(b_1^2+b_2^2)}=
\]
\[
~~~~~~~~~~~~~~~~~~~~~~~~~~~~~~~~~~=\left(\frac{1}{\sqrt{2\pi}}\right)^{3}\cdot
2^{14}\cdot\left(\frac{1}{\sqrt{2m}} \right)^{11} 4^{2m}\cdot
(b_1-b_2)^4\cdot e^{-2(b_1^2+b_2^2)},
\]
which verifies~\eqref{sch3}.
\end{proof}
\end{example}

\section{Asymptotics for the general sums}
As in~\cite[Theorem 3.2]{regev}, Proposition~\ref{sch9} implies
\begin{thm}\label{d.sum1}\cite[Corollary 4.4 corrected]{regev}
Let $\Omega(s)\subset \mathbb{R}^s$ denote the following domain:
\[
\Omega(s)=\{(x_1,\ldots,x_s)\in \mathbb{R}^s\mid x_1\ge\cdots\ge
x_s\quad\mbox{and}\quad x_1+\cdots+x_s=0\}.
\]
Also recall Definition~\ref{definition1}. Then, as $m\to\infty$,
\[
T_{d,ds}^{(\al)}(dm)\simeq~~~~~~~~~~~~~~~~~~~~~~~~~~~~~~~~~~~~~~~~~~~~~~~~~~~~~~~~~~~~~~~~~~~~~~~~~~~~~~~~~~~~~~~~~~
\]
\[
\simeq\left[ \left( \frac{1}{\sqrt{2\pi}} \right)^{ds-1}\cdot
\sqrt d\cdot s^{d^2s^2/2}\cdot (2!\cdots (d-1)!)^s\cdot
\left(\frac{1}{\sqrt{m}} \right)^{(d^2s^2+d^2s-2)/2} \cdot
(ds)^{dm}\right]^\al\cdot
\]
\[~~~~~~~~~~~~~~~~~~~~~~~~~~~~~~~~~~~~~~~~~~~~~~~~~~~~~~~~~~~~~~~~~~~~~~~~~~\cdot
(\sqrt{m})^{s-1} \cdot I(d^2,s,\al)\] where
\[
I(d^2,s,\al)=\int_{\Omega(s)}\left[D_s(x_1,\ldots,x_s]^{d^2}\cdot
e^{-(ds/2)(x_1^2+\cdots +x_s^2)}\right]^\al\cdot dx_1\cdots
dx_{s-1}.
\]
\end{thm}

\begin{remark}
Note that by~\cite[Section 4]{regev} and by the Selberg
integral~\cite{forrester},~\cite{garsia},~\cite{selberg}
\[ I(d^2,s,\al)=\left(\frac{d}{s}\right)^{(s-1)(\al
s+2)/4}\cdot\frac{d}{\sqrt
s}\cdot\sqrt{\frac{\al}{2\pi}}\cdot\frac{1}{s!}\cdot~~~~~~~~~~~~~~~~~~~~~~~~~~~~~~~~~~~~~~~~~~~~~~~~~~~~~~~
\]
\[
~~~~~~~~~~~~~~~~~~~~~~~~~~~~~~~\cdot(2\pi)^{s/2}\cdot(d^2\al)^{-s/2-d^2\al
s(s-1)/4}\cdot (\Gamma(1+d^2\al/2))^{-s}\cdot\prod_{j=1}^s
\Gamma(1+d^2 \al j/2).
\]
\end{remark}
Thus Theorem~\ref{d.sum1} can be rewritten as follows.

\begin{thm}\label{d.sum2}
Let $1\le s,d\in\mathbb{Z}$ and $0<\al\in\mathbb{R}$. Then, as
$m\to\infty$,
\[
T_{d,ds}^{(\al)}(dm)\simeq~~~~~~~~~~~~~~~~~~~~~~~~~~~~~~~~~~~~~~~~~~~~~~~~~~~~~~~~~~~~~~~~~~~~~~~~~~~~~~~~~~~~~~~~~~
\]
\[
\simeq\left[ \left( \frac{1}{\sqrt{2\pi}} \right)^{ds-1}\cdot
\sqrt d\cdot s^{d^2s^2/2}\cdot (2!\cdots (d-1)!)^s\cdot
\left(\frac{1}{\sqrt{m}} \right)^{(d^2s^2+d^2s-2)/2} \cdot
(ds)^{dm}\right]^\al\cdot
\]
\[~~~~~~~~~~~~~~~~~~~\cdot
(\sqrt{m})^{s-1} \cdot\left(\frac{d}{s}\right)^{(s-1)(\al
s+2)/4}\cdot\frac{d}{\sqrt
s}\cdot\sqrt{\frac{\al}{2\pi}}\cdot\frac{1}{s!}\cdot~~~~~~~~~~~~
\]
\[
~~~~~~~~~~~~~~~~~~~~~~~~~~~~~~~\cdot(2\pi)^{s/2}\cdot(d^2\al)^{-s/2-d^2\al
s(s-1)/4}\cdot (\Gamma(1+d^2\al/2))^{-s}\cdot\prod_{j=1}^s
\Gamma(1+d^2 \al j/2).
\]

\end{thm}

\subsection{Some special cases}\label{3.1}
\subsubsection{The case $s=1$}
Let $s=1$. In that case $B_{d,d}(dm)=\{\lm\}$ where
$\lm=(m,\ldots,m)=(m^d)$. Thus, for Theorem~\ref{d.sum1} to hold,
the product of the factors after the factor $[...]^\al$ should
equal 1, which is easy to verify.

\subsubsection{The sums $S_s^{(\al)}(m)$}
In the  case $d=1$, in the notations of~\cite{regev},
$T_{1,s}^{(\al)}(m)=S_s^{(\al)}(m)$, and Theorem~\ref{d.sum2}
becomes
\begin{thm}\label{d.sum02}~\cite[Corollary 4.4]{regev}.
Let $d=1$, $1\le s\in\mathbb{Z}$, $0\le \al\in\mathbb{R}$. Then,
as $m\to\infty$,
\[
T_{1,s}^{(\al)}(m)=S_s^{(\al)}(m)\simeq~~~~~~~~~~~~~~~~~~~~~~~~~~~~~~~~~~~~~~~~~~~~~~~~~~~~~~~~~~~~~~~~~~~~~~~~~~~~~~~
\]
\[
\simeq\left[
%%%%%%%%%%%%%
\left( \frac{1}{\sqrt{2\pi}} \right)^{s-1}\cdot s^{s^2/2}\cdot
\left(\frac{1}{\sqrt{m}} \right)^{(s^2+s-2)/2} \cdot
s^{m}\right]^\al\cdot (\sqrt{m})^{s-1}
\cdot\left(\frac{1}{s}\right)^{(s-1)(\al
s+2)/4}\cdot\frac{1}{\sqrt
s}\cdot\sqrt{\frac{\al}{2\pi}}\cdot\frac{1}{s!}\cdot
\]
\[
~~~~~~~~~~~~~~~~~~~~~~~~~~~~~~~~~~~~~~~~~~~~~~~\cdot(2\pi)^{s/2}\cdot
\al^{-s/2-\al s(s-1)/4}\cdot
(\Gamma(1+\al/2))^{-s}\cdot\prod_{j=1}^s \Gamma(1+\al j/2).
\]
This agrees with the asymptotic value of $S_s^{(\al)}(m)$ as given
by~\cite[Corollary 4.4]{regev} in the case $d=1$.
\end{thm}

\subsubsection{The case $d=1$ and $\al=1$}

\begin{thm}\label{d.sum02}
Let $d=\al=1$, then as $m\to\infty$,
\[
T_{1,s}^{(1)}(m)\simeq \left( \frac{1}{\sqrt{2\pi}}
\right)^{s-1}\cdot s^{s^2/2}\cdot \left(\frac{1}{\sqrt{m}}
\right)^{(s^2+s-2)/2} \cdot s^{m}\cdot (\sqrt{m})^{s-1}
\cdot\left(\frac{1}{s}\right)^{(s-1)( s+2)/4}\cdot\frac{1}{\sqrt
s}\cdot\sqrt{\frac{1}{2\pi}}\cdot\frac{1}{s!}\cdot
\]
\[
~~~~~~~~~~~~~~~~~~~~~~~~~~~~~~~~~~~~~~~~~~~~~~~\cdot(2\pi)^{s/2}
\cdot (\Gamma(1+1/2))^{-s}\cdot\prod_{j=1}^s \Gamma(1+ j/2)=
\]
\[
=(\sqrt s)^{s(s-1)/2}\cdot \frac{1}{s!}\cdot\left(\frac{1}{\sqrt
m} \right)^{s(s-1)/2}\cdot
s^m\cdot(\Gamma(1+1/2))^{-s}\cdot\prod_{j=1}^s \Gamma(1+ j/2),
\]
which agrees with~\cite[(F.4.5.1)]{regev}.
\end{thm}

\subsubsection{The case $d=1,~\al=2$}
Consider the case $d=1$ and $\al=2$ (any $s$), then
\[
T_{1,s}^{(2)}(n)\simeq\left(\frac{1}{\sqrt{2\pi}}
\right)^{s-1}\cdot\left(\frac{1}{\sqrt{2}} \right)^{s^2-1}\cdot
(\sqrt s)^{s^2}\cdot  2!\cdots (s-1)!\cdot
\left(\frac{1}{\sqrt{n}} \right)^{s^2-1} \cdot s^{2n}.
\]
For example, when $s=2$ we have
\[
T_{1,2}^{(2)}(n)\simeq \frac{1}{\sqrt{\pi}}\cdot\frac{1}{n\sqrt
n}\cdot 4^n.
\]
In this case we know~\cite[page 64]{knuth} that
$T_{1,2}^{(2)}(n)=(2n)!/(n!\cdot (n+1)!)=C_n$, the $n$-th Catalan
number, and by applying Stirling's formula directly, we obtain the
same asymptotic value.

\subsubsection{The case $s=d=2$ and $\al=1$}
The case $s=d=2$ and $\al=1$.  By Theorem~\ref{d.sum2}
\[
T_{2,4}^{(1)}(2m)\simeq \left[ \left( \frac{1}{\sqrt{2\pi}}
\right)^3\cdot \sqrt 2\cdot 2^{8}\cdot \left(\frac{1}{\sqrt{m}}
\right)^{11} \cdot (4)^{2m}\right]\cdot (\sqrt{m}) \cdot
\frac{2}{\sqrt 2}\cdot\sqrt{\frac{1}{2\pi}}\cdot\frac{1}{2} \cdot
2\pi\cdot 4^{-3}\cdot \frac{2!\cdot 4!}{2!\cdot 2!}=
\]
\[
~~~~~~~~~~~~~~~~~~~~~~~~~~~~~~~~~~~~~~~~~~~~~~~~~~~~~=\frac{1}{\pi}\cdot
24\cdot\left(\frac{1}{m} \right)^5\cdot 4^{2m}.
\]

Note that sequence A005700 of~\cite{sloane} gives the following
remarkable identity:
\begin{eqnarray}\label{sof1}
T_{2,4}^{(1)}(2m)=\frac{6\cdot (2m)!\cdot(2m+2)!}{m!\cdot
(m+1)!\cdot(m+2)!\cdot(m+3)!}.
\end{eqnarray}
Applying Stirling's formula to the right-hand-side of~\eqref{sof1}
we obtain the same asymptotic value:
\[
\frac{6\cdot (2m)!\cdot(2m+2)!}{m!\cdot
(m+1)!\cdot(m+2)!\cdot(m+3)!}\simeq\frac{1}{\pi}\cdot
24\cdot\left(\frac{1}{m} \right)^5\cdot 4^{2m},
\]
thus verifying Theorem~\ref{d.sum2} in this case.

A. Regev, Department of Mathematics, The Weizmann Institute of
Science, Rehovot 76100,  Israel

e-mail: amitai.regev~at~weizmann.ac.il

\end{document}